\documentclass[a4paper,11pt]{article}
\usepackage[all]{xy}
\usepackage[english]{babel}
\usepackage[latin1]{inputenc}
\usepackage{amsfonts}
\usepackage{amsthm}
\usepackage{amsmath}
\usepackage{amsfonts}
\usepackage{latexsym}
\usepackage{amssymb}
\usepackage{mathrsfs}
\usepackage[usenames]{color}

\newtheorem{fed}{Definition}[section]
\newtheorem{teo}[fed]{Theorem}
\newtheorem{cor}[fed]{Corollary}
\newtheorem*{teo*}{Theorem}

\newtheorem{lem}[fed]{Lemma}
\newtheorem{pro}[fed]{Proposition}

\theoremstyle{definition}
\newtheorem{rem}[fed]{Remark}

\newtheorem{exa}[fed]{Example}

\newtheorem*{teoD}{Theorem (Douglas)}

\def\bdem{\begin{proof}}
\def\edem{\renewcommand{\qed}{\hfill $\blacksquare$}
\end{proof}}

\date{}
\begin{document}

\title{Antitonicity property of the Moore-Penrose inverse of  selfadjoint operators}

\author{Guillermina Fongi $^{a}$, M. Celeste Gonzalez $^{b}$ $^{c}$\\ 
\fontsize {9}{9} \selectfont{$^a$ 
Centro Internacional Franco Argentino de Ciencias de la Informaci\'on y de Sistemas, CIFASIS (CONICET-UNR)} \\ \selectfont \fontsize {9}{9} \selectfont{Ocampo y Esmeralda (2000)  Rosario, Argentina.}
\\
\fontsize {9}{9} \selectfont{$^b$ Instituto Argentino de Matem\'atica ``Alberto P. Calder\'on'', IAM-CONICET} \\ \selectfont \fontsize {9}{9} \selectfont{Saavedra 15, Piso 3 (1083), Buenos Aires, Argentina.}
\\
\fontsize {9}{9} \selectfont{$^c$ Instituto de Ciencias, Universidad Nacional de General Sarmiento, Argentina.}\\
\fontsize {9}{9} \selectfont{$^a$  gfongi@conicet.gov.ar, $^b$ celeste.gonzalez@conicet.gov.ar}}
\date{}
\maketitle



\begin{abstract}
In this article we  study  the antitonicity property of the Moore-Penrose inverse in the class of selfadjoint operators, with respect to the L\"owner order.  For this purpose, we  employ different  positive decompositions that   selfadjoint operators admit.  In addition, we relate  a weak version of the antitonicity property with Thompson components.
\end{abstract}

\section{Introduction}

It is well known that, given two positive invertible, bounded,  linear operators $S, T $ on a Hilbert space  the following relation holds:
$$
0< S \leq T \quad \text{if and only if} \quad T^{-1} \leq S^{-1},
$$
where $\leq$ denotes the L\"owner order defined on the set of selfadjoint operators. This property is commonly referred to as the \textit{antitonicity}  or \textit{antimonotonicity} property.

Due to 
the central role of the Moore-Penrose inverse in operator theory and its diverse applications (such as approximation theory, control theory, and image processing), its algebraic and topological properties have attracted significant attention. Among these, antitonicity stands out for its fundamental connection to order structures.
The first investigations into extending the antitonicity property were primarily confined to the finite-dimensional setting, focusing specifically on selfadjoint matrices \cite{MA, MR1048801, Liski}. The techniques employed in these initial proofs relied heavily on matrix-specific properties.  Later, the study was extended to linear selfadjoint relations and (not necessarily bounded) selfadjoint operators on Hilbert spaces, under specific hypotheses regarding the inertia of the involved relations or operators \cite{BHdSW-square integrable, BHdSW-limit, HassiNor1993Antitonicity,BHWDSAntitonicity}.

The antitonicity property of the Moore-Penrose inverse of positive semidefinite closed range operators was studied in \cite{FG-MoorePenrose y orden} and applied to the convergence of iterative processes in \cite{FG-splitting}. In the present paper, we study this property for bounded selfadjoint operators on an infinite-dimensional Hilbert space.

In \cite[Theorem 3.8]{BHWDSAntitonicity}, an antitonicity criterion for Moore-Penrose inverses is obtained for semibounded selfadjoint operators on a separable Hilbert space under the assumption $i^-(H_1)+i^0(H_1)<\infty$. Here $i^-(H_1)$ is the dimension of the negative spectral subspace and $i^0(H_1)$ is the dimension of the nullspace. Our approach instead uses positive decompositions and geometric conditions on nullspace-range intersections, without imposing such finiteness assumptions. As an application, we characterize the Thompson component of a closed range selfadjoint operator by a weak version of antitonicity.

The paper is organized as follows. Section 3 is devoted to the positive orthogonal decomposition of a selfadjoint operator. Here, we analyze its distinctive characteristics compared to other positive decompositions that a selfadjoint operator may admit. Section 4 contains the main results of the paper, namely Theorems   \ref{antitonicityHermitianos} and \ref{antitonicity-igualdad de rangos} and Proposition \ref{Thompson component}. In this section, we characterize the validity of the antitonicity property of the Moore-Penrose inverse of
bounded linear selfadjoint operators, employing  positive decompositions as a key tool. We also obtain equalities between the ranges of suitable square roots and between sums of positive and negative ranges. In addition, we  provide a new description of the Thompson component of a closed range selfadjoint operator.

\section{Preliminaries}

 Throughout, we denote by $\mathcal{H}$ a complex Hilbert space with inner product $\langle \cdot , \cdot \rangle$ and by $\mathcal{L}(\mathcal{H})$ the algebra of bounded linear operators from $\mathcal{H}$ to $\mathcal{H}$. Given $T\in\mathcal{L}(\mathcal{H})$, by $\mathcal{R}(T)$ and $\mathcal{N}(T)$ we denote the range and nullspace of $T$, respectively. Also, $T^*$ indicates the adjoint operator of $T\in\mathcal{L}(\mathcal{H})$. The set of selfadjoint  operators of $\mathcal{L}(\mathcal{H})$ is denoted by $\mathcal{L}^s$ and $\mathcal L_{cr}^s$ stands for the subset of $\mathcal{L}^s$ consisting of operators with closed range. Given two operators $S,T\in\mathcal{L}^s$ it is said that $S$ is lower than $T$  with respect to the L\"owner order, $S\leq T$, if $\langle Sx,x\rangle\leq \langle Tx,x\rangle$ for all $x\in\mathcal{H}$. An operator $T\in\mathcal{L}^s$ is positive  semidefinite, $T\geq 0$, if $\langle Tx,x\rangle \geq 0$ for all $x\in\mathcal{H}$. The cone of the  positive  semidefinite operators of $\mathcal{L}^s$ is $\mathcal{L}^+$. In addition, $\mathcal L_{cr}^+$ denotes the subset of $\mathcal{L}^+$ consisting of operators with closed range.  Given $T\in\mathcal{L}(\mathcal{H})$, $|T|$ is the operator in $\mathcal{L}^+$ defined by $|T|:=(T^*T)^{1/2}$. The orthogonal sum between two subspaces $\mathcal{V}$ and $\mathcal{W}$ is denoted by $\mathcal{V}\oplus\mathcal{W}$. In addition, if $\mathcal{V}\subseteq \mathcal{H}$ is a closed subspace then  by $P_{\mathcal{V}}$ we denote the orthogonal projection onto $\mathcal{V}$; for every $T\in\mathcal{L}(\mathcal{H})$, we write $P_T:=P_{\overline{\mathcal{R}(T)}}$.  On the other hand, the symbol $\oplus$  will also be used to indicate the Hilbert space direct sum of Hilbert spaces. In this context, if $\mathcal H$, $\mathcal K$ are Hilbert spaces, $S \in \mathcal L(\mathcal H)$ and $T \in \mathcal L(\mathcal K)$ then $S\oplus T\in \mathcal{L}(\mathcal H\oplus \mathcal K)$ is defined by $(S\oplus T)(x\oplus y)=Sx\oplus Ty\in \mathcal H\oplus \mathcal K$.

The next result on range inclusion and factorization will be useful in this article. Its proof can be found in \cite{MR0203464}. 

\begin{teoD}
	Let $S,T \in \mathcal{L}(\mathcal{H})$. The following conditions are equivalent:
	\begin{enumerate}
		\item $\mathcal{R}(T)\subseteq \mathcal{R}(S)$;
		\item there exists a number $\lambda >0$ such that $TT^*\leq \lambda SS^*$;
		\item there exists $C\in\mathcal{L}(\mathcal{H})$ such that $SC=T$.
	\end{enumerate}
	In addition, if any of the above conditions holds then there exists a unique operator $X_r\in\mathcal{L}(\mathcal{H})$ such that $SX_r=T$ and $\mathcal{R}(X_r)\subseteq \mathcal{N}(S)^\bot$. The operator $X_r$ is called the \textit{Douglas reduced solution} of $SX=T$.
\end{teoD} 

The following result, which asserts that the sum of two operator ranges is the range of a positive operator, is due to Crimmins. Its proof can be found in \cite[Theorem 2.2]{FW}.

\begin{teo}\label{Crimmins}
Consider $T, S\in \mathcal L(\mathcal H)$. 
Then 
$\mathcal R(T)+\mathcal R(S)=\mathcal R(\sqrt{TT^*+SS^*}).
$
\end{teo}

Different classes of pseudoinverses of  operators  in $\mathcal L(\mathcal H)$  are studied  in the literature. Among them, the Moore-Penrose inverse of $T \in \mathcal{L}(\mathcal{H})$ is the unique linear densely defined operator $T^\dagger: \mathcal{R}(T)\oplus\mathcal{R}(T)^\bot\rightarrow \mathcal{H}$ which satisfies, simultaneously, the following four equations:
\begin{equation}\label{MPequations}
TXT=T, \  XTX=X, \ TX=P_T|_{\mathcal{R}(T)\oplus\mathcal{R}(T)^\bot}, \ XT=P_{T^*}.
\end{equation}
It is well known  that $T^\dagger\in\mathcal{L}(\mathcal{H})$ if and only if $\mathcal{R}(T)$ is closed. For a closed range operator $T$ in $\mathcal{L}(\mathcal{H})$, the Moore-Penrose inverse can also be defined as the reduced solution of the operator equation $TX=P_T$, see \cite{arias2008generalized}. For a positive operator $T$, the range of $T$ is closed if and only if the range of $T^{1/2}$ is closed. In this case, $
\mathcal R(T)=\mathcal R(T^{1/2})=\mathcal R(T^\dagger)
=\mathcal R\bigl((T^\dagger)^{1/2}\bigr) $ and $\mathcal N(T^\dagger)=\mathcal N(T).
$

The \textit{reverse order law}, $(ST)^\dagger=T^\dagger S^\dagger$, has been extensively studied for closed range operators $S,T\in\mathcal L(\mathcal H)$ whose product also has closed range. The following characterization can be found  in Greville's  \cite{Gre} for matrices; extensions to Hilbert space operators were established by Bouldin \cite{Bouldin1973} and Izumino \cite{Izumino1982}.

\begin{pro}\label{ROL-Gre}
Let $S,T\in\mathcal{L}(\mathcal{H})$ with closed range such that $ST$ has closed range. Then $(ST)^\dagger=T^\dagger S^\dagger$ if and only if $\mathcal{R}(S^*ST)\subseteq \mathcal{R}(T)$ and $\mathcal{R}(TT^*S^*)\subseteq \mathcal{R}(S^*)$.
\end{pro}

\section{On the positive orthogonal decomposition}

In this section we focus on the positive orthogonal decomposition of selfadjoint operators and its connection to other decompositions involving differences of positive operators. 

\begin{lem}
Let $T\in\mathcal{L}(\mathcal{H}
)$. Then $T\in\mathcal{L}^s$ if and only if $T=T_1-T_2$, where $T_1, T_2\in\mathcal{L}^+$.  
\end{lem}

\begin{proof}
If $T\in\mathcal{L}^s$ then $T=\frac{|T|+T}{2}-\frac{|T|-T}{2}$ and the assertion follows by taking $T_1=\frac{|T|+T}{2}\in\mathcal{L}^+$ and $T_2=\frac{|T|-T}{2}\in\mathcal{L}^+$. The converse is immediate.
\end{proof}

Given a selfadjoint operator $T\in\mathcal{L}(\mathcal{H})$, the decomposition $T=T^+-T^- 
$, with $
T^+=\frac{|T|+T}{2}$ and $  T^-=\frac{|T|-T}{2}
$  is called {\it{the positive orthogonal decomposition}} of $T$. 
It is clear that $TT^+=T^+T$,  $TT^-=T^-T$, $T\leq T^+$ and $-T\leq T^-$. Moreover, it holds that $\mathcal{R}(T)=\mathcal{R}(T^+)\oplus \mathcal{R}(T^-)$, see \cite{AG-aditividad, FM-desc.positivas}. 
If $T\in\mathcal L_{cr}^s$, then $T^+$ and $T^-$ also have closed ranges. This observation will be used throughout Section~4.

Other types of positive decompositions for selfadjoint operators have been explored in the literature. For instance, in \cite{FM-desc.positivas}  decompositions $T=T_1-T_2$ with  $T_1, T_2 \in \mathcal{L}^+$ and $\mathcal{R}(T) = \mathcal{R}(T_1) \dot+ \mathcal{R}(T_2)$ are investigated,  where $\dot+$ denotes the direct sum of the subspaces. 

Below, we compare the positive orthogonal decomposition of $T \in \mathcal{L}^s$ with decompositions $T = T_1 - T_2$ satisfying $T_i \in \mathcal{L}^+$ and $T_i T \in \mathcal{L}^s$ for $i = 1, 2$. The following standard fact will be useful; see \cite{RN}.

\begin{lem}\label{mayorizacion de modulo de T}
Let $T\in\mathcal{L}^s$ and $C\in\mathcal{L}^+$. If $CT=TC$ and $C\geq \pm T$ then $C\geq |T|$.
 \end{lem}

 \begin{proof}
 Let $T=T^+-T^-$ be the positive orthogonal decomposition of $T$. Then $P_{\mathcal{N}(T^+)}T=TP_{\mathcal{N}(T^+)}$ and $P_{\mathcal{N}(T^+)}C=CP_{\mathcal{N}(T^+)}$. In fact, it is clear that $P_{\mathcal{N}(T^+)}T=-T^-$ and so, $P_{\mathcal{N}(T^+)}T=TP_{\mathcal{N}(T^+)}$. Also, since $CT=TC$ then $C|T|=|T|C$. Hence $CT^+=T^+C$ and so that  $T^+CP_{\mathcal{N}(T^+)}=CT^+P_{\mathcal{N}(T^+)}=0$. Therefore $\mathcal{R}(CP_{\mathcal{N}(T^+)})\subseteq \mathcal{N}(T^+)$ and then $CP_{\mathcal{N}(T^+)}=P_{\mathcal{N}(T^+)}CP_{\mathcal{N}(T^+)}\in\mathcal{L}^+$. Hence $CP_{\mathcal{N}(T^+)}=P_{\mathcal{N}(T^+)}C$, so that $CP_{T^+}=P_{T^+}C$. Therefore, since $C\geq \pm T$ then $CP_{T^+}\geq TP_{T^+}=T^+$ and $CP_{\mathcal{N}(T^+)}\geq -TP_{\mathcal{N}(T^+)}=T^-$. So that, $C\geq T^++T^-=|T|$.
 \end{proof}

Next, we will see that the positive orthogonal decomposition of $T$ is minimal among all decompositions $T=T_1-T_2$ of $T$, where $T_1, T_2\in\mathcal{L}^+$, $TT_1=T_1T$ and $TT_2=T_2T$.

 \begin{pro}\label{dpo minimalidad}
     Consider $T\in \mathcal L^s$  with positive orthogonal decomposition $T=T^+-T^-$,  then 
\begin{align}
T^+&=\min\{T_1\in\mathcal L^+:T=T_1-T_2,\ T_2\in\mathcal L^+,\ T_1T=TT_1\};\label{T1}\\
T^-&=\min\{T_2\in\mathcal L^+:T=T_1-T_2,\ T_1\in\mathcal L^+,\ T_2T=TT_2\};\label{T2}\\
|T|&=\min\{C\in\mathcal L^+:
C=T_1+T_2,\ T=T_1-T_2,
T_1,T_2\in\mathcal L^+,\ CT=TC
\}.\label{T3}
\end{align}
All minima are taken with respect to the L\"owner order.
     Moreover, $T=T_1-T_2$, where $T_1, T_2\in\mathcal{L}^+$, $TT_1=T_1T$ and $TT_2=T_2T$ if and only if $T_1=T^++Z$ and $T_2=T^-+Z$, where $Z\in\mathcal{L}^+$ and $ZT=TZ$.
 \end{pro}

\begin{proof}
Suppose that $T=T_1-T_2$ with $T_1,T_2\in\mathcal L^+$, and set $C=T_1+T_2$. Since
$
T_1=\frac{C+T}{2}$, $T_2=\frac{C-T}{2},
$ then 
commutation of any one of $T_1,T_2,C$ with $T$ is equivalent to commutation of the other two with $T$. Assume that these equivalent commutation conditions hold. Moreover, $C-T=2T_2\geq0$ and $C+T=2T_1\geq0$. Thus Lemma~\ref{mayorizacion de modulo de T} gives $C\geq|T|$. Consequently,
$
Z:=\frac{C-|T|}{2}\geq0,  ZT=TZ,
$ and
$
T_1=\frac{C+T}{2}=T^++Z,
T_2=\frac{C-T}{2}=T^-+Z.
$
Conversely, if $Z\in\mathcal L^+$ commutes with $T$, then $T_1=T^++Z$ and $T_2=T^-+Z$ are positive, commute with $T$, and satisfy $T=T_1-T_2$. This proves the asserted parametrization. It also gives $T_1\geq T^+$, $T_2\geq T^-$ and $C=|T|+2Z\geq|T|$. Since $Z=0$ is admissible, the minima in \eqref{T1}--\eqref{T3} are attained at $T^+$, $T^-$ and $|T|$, respectively.
\end{proof}

\section{Antitonicity property of the Moore-Penrose inverse}

In this section we study  the antitonicity property of the Moore-Penrose inverse of selfadjoint operators, by means of positive decompositions of the operators involved. The next result  will be useful. 

\begin{lem}\label{dpo de MP}
Consider $S,T\in\mathcal L(\mathcal H)$ with closed ranges such that
$
\mathcal R(S)\bot\mathcal R(T)$ and $\mathcal R(S^*)\bot\mathcal R(T^*). $ Then
$(S+T)^\dagger=S^\dagger+T^\dagger.$
\end{lem}

\begin{proof}
Since $
\mathcal R(S)\bot\mathcal R(T)$ and $\mathcal R(S^*)\bot\mathcal R(T^*)$, it holds that 
$\mathcal R(S+T)=\mathcal R(S)\oplus\mathcal R(T)$ is closed. The assertion then follows by verifying the four Moore-Penrose equations \eqref{MPequations} for $S^\dagger+T^\dagger$.
\end{proof}

Among pioneering works on the antitonicity property of the Moore-Penrose inverse in the matrix context, we can cite  \cite{MA, MR1048801}. Later, in \cite{FG-MoorePenrose y orden} we studied this property for positive closed range operators in $\mathcal{L(\mathcal{H})}$. For completeness, we give a proof of \cite[Theorem 4.1]{FG-MoorePenrose y orden}.

\begin{lem}\label{antitonicitypositivos}
Let $S,T\in\mathcal{L}^+_{cr}$.  Then, any two of the following conditions imply the third condition:
\begin{enumerate}
\item $S\leq T$;
\item $T^\dagger\leq S^\dagger$;
\item $\mathcal{R}(S)\cap\mathcal{N}(T)=\mathcal{R}(T)\cap\mathcal{N}(S)=\{0\}$.
\end{enumerate}
\end{lem}

\begin{proof}
Assume first that conditions 1 and 2 hold. By Douglas' theorem and since $S, T\in\mathcal{L}^+_{cr}$ then $
\mathcal{R}(S)=\mathcal R(S^{1/2})\subseteq\mathcal R(T^{1/2})=\mathcal R(T)$
 and $
\mathcal R(T)=\mathcal R\bigl((T^\dagger)^{1/2}\bigr)
\subseteq\mathcal R\bigl((S^\dagger)^{1/2}\bigr)=\mathcal R(S).
$
Hence $\mathcal R(S)=\mathcal R(T)$, and condition 3 follows.
Suppose now that conditions 1 and 3 hold. Douglas' theorem again gives $\mathcal R(S)\subseteq\mathcal R(T)$. Since these subspaces are closed,
$
\mathcal R(T)=\mathcal R(S)\oplus
\bigl(\mathcal R(T)\cap\mathcal N(S)\bigr)=\mathcal R(S).
$
Let $\mathcal M=\mathcal R(S)=\mathcal R(T)$. If $\mathcal M=\{0\}$, the conclusion is immediate. Otherwise, the restrictions $S|_{\mathcal M}$ and $T|_{\mathcal M}$ are positive invertible operators on $\mathcal M$, and $S|_{\mathcal M}\leq T|_{\mathcal M}$. The usual antitonicity of inversion gives $(T|_{\mathcal M})^{-1}\leq (S|_{\mathcal M})^{-1}$. Extending both inverses by zero on $\mathcal M^\perp$ yields $T^\dagger\leq S^\dagger$, which is condition 2.
Finally, the proof that items 2 and 3 imply item 1 is similar to the above one.
\end{proof}

We now apply Lemma~\ref{antitonicitypositivos} to positive decompositions of selfadjoint operators. The resulting
comparison
yields antitonicity for the operators themselves when the decompositions are orthogonal.
\begin{teo} \label{antitonicityHermitianos}
Let $S,T\in\mathcal{L}^s$ and let $S=S_1-S_2$ and $T=T_1-T_2$ where $S_i, T_i \in \mathcal L^+_{cr}$, for $i=1,2$. Suppose that $\mathcal{R}(T_1)\bot\mathcal{R}(S_2)$ and $\mathcal{R}(T_2)\bot\mathcal{R}(S_1)$. Then, any two of the following conditions imply the third condition:
\begin{enumerate}
\item $S\leq T$;
\item $T_1^\dagger -T_2 ^\dagger \leq S_1^\dagger-S_2^\dagger$;
\item $\mathcal{R}(S_1+T_2)\cap\mathcal{N}(T_1+S_2)=\mathcal{R}(T_1+S_2)\cap\mathcal{N}(S_1+T_2)=\{0\}$.
\end{enumerate}
\end{teo}

\begin{proof}
Set $A=S_1+T_2$ and $B=T_1+S_2$. Since $\mathcal R(T_2)\perp\mathcal R(S_1)$ and $\mathcal R(T_1)\perp\mathcal R(S_2)$, then $A$ and $B$ have  closed ranges. By Lemma~\ref{dpo de MP},
$
A^\dagger=S_1^\dagger+T_2^\dagger,$ and $
B^\dagger=T_1^\dagger+S_2^\dagger.
$
Condition 1 is equivalent to $A\leq B$, condition 2 is equivalent to $B^\dagger\leq A^\dagger$, and condition 3 is 
$\mathcal R(A)\cap\mathcal N(B)=\mathcal R(B)\cap\mathcal N(A)=\{0\}$.
Thus the assertion follows directly from Lemma~\ref{antitonicitypositivos}.
\end{proof}

Note that  $S$ and $T$ in Theorem~\ref{antitonicityHermitianos} are not necessary closed range operators; only their positive components are assumed to have closed range. For arbitrary positive decompositions, $S_1^\dagger-S_2^\dagger$ need not coincide with $S^\dagger$, even when $S$ has closed range.

\begin{exa}
Let us see that there exist operators $S$ and $T$ which satisfy the hypotheses of Theorem \ref{antitonicityHermitianos}. In fact,  consider $\{e_1, e_2, \cdots\}$ the canonical orthonormal basis of $\ell^2_{\mathbb{N}}$ and the subspaces $\mathcal{M}_1=\overline{\operatorname{span}\{e_1,e_3,e_5,\ldots\}}$ and $\mathcal{M}_2=\overline{\operatorname{span}\{e_1+e_2,e_3+e_4,e_5+e_6,\ldots\}}$. Note that $\mathcal{M}_1$ and $\mathcal{M}_2$ are not orthogonal because $\langle e_1, e_1+e_2\rangle\neq 0$. Take $\mathcal{N}_1=\mathcal{M}_2^\bot$ and $\mathcal{N}_2=\mathcal{M}_1^\bot$ and consider $S=P_{\mathcal{M}_1}- P_{\mathcal{M}_2}$ and $T=2P_{\mathcal{N}_1}- P_{\mathcal{N}_2}$. Then $S$ and $T$ satisfy the hypotheses of Theorem \ref{antitonicityHermitianos}.
\end{exa}

In  particular, if  the positive orthogonal decompositions of $S$ and $T$  are considered  in the above theorem, we obtain the following result.

\begin{cor} \label{antitonicityhermitianospod}
Let $S,T\in\mathcal L_{cr}^s$ and let $S=S^+-S^-$ and $T=T^+-T^-$ be their positive orthogonal decompositions. Suppose that $\mathcal{R}(T^+)\bot\mathcal{R}(S^-)$ and $\mathcal{R}(T^-)\bot\mathcal{R}(S^+)$. Then, any two of the following conditions imply the third condition:
\begin{enumerate}
\item $S\leq T$;
\item $T^\dagger\leq S^\dagger$;
\item $\mathcal{R}(S^++T^-)\cap\mathcal{N}(T^++S^-)=\mathcal{R}(T^++S^-)\cap\mathcal{N}(S^++T^-)=\{0\}$.
\end{enumerate}
\end{cor}

\begin{proof} 
 Observe that
$
T^\dagger=(T^+)^\dagger-(T^-)^\dagger$ and $
S^\dagger=(S^+)^\dagger-(S^-)^\dagger,
$
which follow from Lemma~\ref{dpo de MP} and the orthogonality of the positive and negative parts.
Then the assertions are consequence of  Theorem~\ref{antitonicityHermitianos} applied to  the positive orthogonal decompositions. 
\end{proof}

In the following remarks and the next two results, we study  the hypotheses of Theorem \ref{antitonicityHermitianos}  and Corollary \ref{antitonicityhermitianospod}.

\begin{rem}
\label{antitonicity y rol}
    Consider $S, T \in \mathcal L_{cr}^s$  with positive orthogonal decompositions $S=S^+-S^-$ and $T=T^+-T^-$. Then $\mathcal{R}(T^+)\bot\mathcal{R}(S^-)$,  $\mathcal{R}(T^-)\bot\mathcal{R}(S^+)$ and $\mathcal R(S^++T^-)=\mathcal R(T^++S^-)$ if and only if
$
\mathcal R(S^+)=\mathcal R(T^+)$ and $\mathcal R(S^-)=\mathcal R(T^-).$
In fact, the orthogonality conditions and $\mathcal R(S^++T^-)=\mathcal R(T^++S^-)$  are equivalent to  
$
\mathcal R(S^+)\oplus\mathcal R(T^-)
=\mathcal R(T^+)\oplus\mathcal R(S^-).
$
Let $x\in\mathcal R(S^+)$. Write $x=y+z$, where
$y\in\mathcal R(T^+)$ and $z\in\mathcal R(S^-)$, then we obtain
$0=\langle x,z\rangle=\|z\|^2$. Hence $x\in\mathcal R(T^+)$. Interchanging $S$ and $T$ gives the reverse inclusion. Similarly, it can be checked that $
\mathcal R(S^-)=\mathcal R(T^-).
$ The converse follows easily.
    \end{rem}

\begin{rem} Observe that given $S,T\in\mathcal{L}^s$ with orthogonal ranges, then $S,T$ satisfy the orthogonality conditions of Corollary~\ref{antitonicityhermitianospod}. If their ranges are also closed, all the hypotheses of that corollary hold. Indeed, if $S=S^+-S^-$ and $T=T^+-T^-$ are the positive orthogonal decompositions of $S$ and $T$, then $\mathcal R(S^+)\oplus \mathcal R(S^-)=\mathcal R(S)\subseteq \mathcal R(T)^\perp=(\mathcal R(T^+) \oplus \mathcal R(T^-))^\perp= \mathcal R(T^+)^\perp \cap \mathcal R(T^-)^\perp$, so that $\mathcal R(S^+)\subseteq \mathcal R(T^-)^\perp$. Similarly, $\mathcal R(S^-)\subseteq \mathcal R(T^+)^\perp$.
However, the converse does not hold. In fact, let $I$ be the identity operator in $\mathcal{L}(\mathcal{H})$ and consider $S,T\in\mathcal{L}(\mathcal{H}\oplus\mathcal{H}\oplus\mathcal{H})$ defined by $S=0\oplus (-4 I )\oplus I
$ and $T=I\oplus (-2 I) \oplus 0
$. It is clear that the ranges of $S$ and $T$ are not orthogonal. Now, from the positive orthogonal decompositions $S=S^+-S^-=0\oplus 0 \oplus I-(0\oplus 4 I \oplus 0)$ and $T=T^+-T^-=I\oplus 0 \oplus 0-(0\oplus 2 I \oplus 0)$ it follows that $\mathcal{R}(T^+)\bot\mathcal{R}(S^-)$ and $\mathcal{R}(T^-)\bot\mathcal{R}(S^+)$.
\end{rem}

\begin{pro}\label{hipotesis para dpo}
    Consider $S, T\in \mathcal L^s$  such  that $ST=TS$,  with positive orthogonal decompositions $S=S^+-S^-$ and $T=T^+-T^-$. Then $\mathcal R(S^+)\perp \mathcal R(T^-)$ and $\mathcal R(S^-)\perp \mathcal R(T^+)$ if and only if $ST\in \mathcal L^+$.
    \end{pro}

\begin{proof}
Since $S$ and $T$ commute, their positive and negative parts commute with one another. Suppose that $ST\geq0$. Then
$
0\leq S^-STS^-=-(S^-)^3T.
$
Multiplying on both sides by $T^+$ gives
$
0\leq T^+\bigl(-(S^-)^3T\bigr)T^+
=-(S^-)^3(T^+)^3.
$
Since the factors are positive and commute then $S^-T^+\geq0$ and so
$(S^-T^+)^3=(S^-)^3(T^+)^3\geq0$. The preceding inequality therefore implies $(S^-T^+)^3=0$. In consequence $S^-T^+=0$. Interchanging $S$ and $T$ yields $T^-S^+=0$,  hence also $S^+T^-=0$. These identities are equivalent to the required orthogonalities.
Conversely, the orthogonality assumptions imply $S^+T^-=S^-T^+=0$. Thus
$T=(S^+-S^-)(T^+-T^-)=S^+T^++S^-T^-\geq0,
$
because both summands are products of commuting positive operators.
\end{proof}

\begin{cor}
 Let $S, T\in \mathcal L^s$ be such that $ST=TS$ and  consider  $S=S_1-S_2$, $T=T_1-T_2$   positive decompositions that satisfy
    \begin{enumerate}
    \item $SS_i=S_iS$, $TT_i=T_iT$, for $i=1,2$;
    \item $\mathcal R(S_1-S^+) \subseteq \mathcal R(S)\cap \mathcal R(T)^\perp$, 
   $ \mathcal R(T_1-T^+) \subseteq \mathcal R(T)\cap \mathcal R(S)^\perp$.  
    \end{enumerate}
    Then $\mathcal R(S_1)\perp \mathcal R(T_2)$ and $\mathcal R(S_2)\perp \mathcal R(T_1)$ if and only if $ST\in\mathcal{L}^+$.
\end{cor}

\begin{proof}
    First observe that, since $SS_i=S_iS$ and $TT_i=T_iT$, for $i=1,2$, then by Proposition \ref{dpo minimalidad} it holds that $S_1=S^++Z_S$, $S_2=S^-+Z_S$, $T_1=T^++Z_T$, $T_2=T^-+Z_T$, where $Z_S, Z_T \in  \mathcal L^+$ and $Z_SS=SZ_S, Z_TT=TZ_T$.
    Now, suppose that  $ST \in \mathcal L^+$. Then by Proposition \ref{hipotesis para dpo}, it holds that $\mathcal R(S^+)\perp \mathcal R(T^-)$ and $\mathcal R(S^-)\perp \mathcal R(T^+)$. Therefore, $S_1T_2=(S^++Z_S)(T^-+Z_T)=0$ because $S^+T^-=0$, $S^+Z_T=0$, $Z_ST^-=0$ and $Z_SZ_T=0$. Here $\mathcal R(Z_T)\subseteq\mathcal R(S)^\perp$ gives $S^+Z_T=0$, while $\mathcal R(Z_S)\subseteq\mathcal R(T)^\perp$ and selfadjointness give $\mathcal R(T)\subseteq\mathcal N(Z_S)$. Since $\mathcal R(Z_T)\subseteq\mathcal R(T)$, the remaining mixed terms also vanish. Similarly, $S_2T_1=0$.
    Conversely, suppose  $\mathcal R(S_1)\perp \mathcal R(T_2)$ and $\mathcal R(S_2)\perp \mathcal R(T_1)$. Then  $ST=S_1T_1+S_2T_2$. Now, observe that $S_1T_1=S^+T^++S^+Z_T+Z_ST^++Z_SZ_T$. Note that  $\mathcal{R}(Z_T)\subseteq\mathcal{R}(S)^\bot=(\mathcal{R}(S^+)+\mathcal{R}(S^-))^\bot=\mathcal{N}(S^+)\cap\mathcal{N}(S^-)\subseteq \mathcal{N}(S^+)$. Similarly $\mathcal R(Z_S)\subseteq \mathcal N(T^+)$. In addition, as $\mathcal{R}(Z_T)\subseteq\mathcal{R}(T)$ and $\mathcal{R}(Z_S)\subseteq\mathcal{R}(T)^\bot$ then $\mathcal{R}(Z_T)\subseteq\mathcal{R}(T)\subseteq \mathcal{N}(Z_S)$. Hence $S_1T_1=S^+T^+ \in \mathcal L^+$ because $S^+,T^+\in \mathcal L^+$  commute. Similarly, $S_2T_2\in \mathcal L^+.$ Thus $ST\in\mathcal{L}^+$.
    \end{proof}

In \cite[Lemma 2]{Liski}, Liski proved that given $S, T$ Hermitian matrices such that $S \leq T$ and $T^\dagger \leq S^\dagger$, then $\mathcal R(S)=\mathcal R(T)$. Now we present the following result for selfadjoint operators on infinite-dimensional Hilbert spaces.

\begin{teo}\label{antitonicity-igualdad de rangos}
Consider $S,T\in\mathcal L_{cr}^s$ with positive orthogonal decompositions $S=S^+-S^-$ and $T=T^+-T^-$. If $S\leq T$ and $T^\dagger\leq S^\dagger$, then
$
\mathcal R\bigl((T^++S^-)^{1/2}\bigr)
=\mathcal R\bigl((S^++T^-)^{1/2}\bigr)
=\mathcal R(S)+\mathcal R(T),
$
or equivalently,
$
\mathcal R(T^+)+\mathcal R(S^-)
=\mathcal R(S^+)+\mathcal R(T^-)
=\mathcal R(S)+\mathcal R(T).
$
If, in addition,
$\mathcal N(T)\subseteq\mathcal N(S^-)$ and $
\mathcal N(S)\subseteq\mathcal N(T^-)
$
then $\mathcal R(S)=\mathcal R(T)$.
\end{teo}

\begin{proof}
Let 
$A=S^++T^-$, $B=T^++S^-$, $
C=(T^+)^\dagger+(S^-)^\dagger$ and 
$D=(S^+)^\dagger+(T^-)^\dagger$. The two order assumptions are equivalent to $A\leq B$ and $C\leq D$, respectively. Since $S^\pm$ and $T^\pm$ have closed ranges, Theorem~\ref{Crimmins} and the closed-range identities of Section~2 give
$\mathcal R(B^{1/2})
=\mathcal R(T^+)+\mathcal R(S^-)=\mathcal R(C^{1/2})$ and 
$\mathcal R(A^{1/2})
=\mathcal R(S^+)+\mathcal R(T^-)=\mathcal R(D^{1/2}).
$
Douglas' theorem therefore yields
$
\mathcal R(A^{1/2})\subseteq\mathcal R(B^{1/2})
=\mathcal R(C^{1/2})\subseteq\mathcal R(D^{1/2})
=\mathcal R(A^{1/2}).
$
Thus all these ranges coincide and they are equal to both
$\mathcal R(T^+)+\mathcal R(S^-)$ and
$\mathcal R(S^+)+\mathcal R(T^-)$, so it contains the ranges of all four positive and negative parts. Consequently it is equal to $\mathcal R(S)+\mathcal R(T)$. This proves the first two identities. 
For the last assertion, taking orthogonal complements in the additional nullspace inclusions and using closedness of the ranges gives
$
\mathcal R(S^-)\subseteq\mathcal R(T) $ and $
\mathcal R(T^-)\subseteq\mathcal R(S).
$
Hence the range identity just established implies
$
\mathcal R(S)+\mathcal R(T)
=\mathcal R(T^+)+\mathcal R(S^-)\subseteq\mathcal R(T)$, $
\mathcal R(S)+\mathcal R(T)
=\mathcal R(S^+)+\mathcal R(T^-)\subseteq\mathcal R(S).
$ 
It follows that $\mathcal R(S)=\mathcal R(T)$.
\end{proof}

For an example satisfying the hypotheses of the preceding theorem, consider $\mathcal H=\ell^2\oplus\ell^2\oplus\ell^2$ and the operators $S,T\in\mathcal L_{cr}^s(\mathcal H)$ defined by
$
S = 2I \oplus (-2I)\oplus 0$ and $ T =  4I \oplus (-I) \oplus 0$, 
where $I$ denotes the identity operator on $\ell^2$.
Indeed,
$
T-S=2I\oplus I\oplus0\geq0,$ $
S^\dagger-T^\dagger=\tfrac14 I\oplus\tfrac12 I\oplus0\geq0.
$
The nullspace inclusions also hold.

As a consequence of Theorem \ref{antitonicity-igualdad de rangos}, we prove that closed range  positive semidefinite operators with the same range satisfy a weak version of the antitonicity property. The next result can be compared to \cite[Proposition 5.6]{FG-MoorePenrose y orden} which asserts that given  $S,T\in\mathcal{L}^+$ with closed ranges the reverse order law for $ST$ and $TS$ holds if $S\leq T$ and $T^\dagger \leq S^\dagger$.

\begin{cor}\label{weak antitonicity}
Consider $S,T \in \mathcal{L}_{cr}^+$. Then, there exists  $\lambda >0$ such that $S\leq \lambda T$ and $\frac{1}{\lambda}T^\dagger \leq S^\dagger$ if and only if $\mathcal{R}(S)=\mathcal{R}(T)$.
If any of the above conditions holds then $(ST)^\dagger=T^\dagger S^\dagger$ and $(TS)^\dagger=S^\dagger T^\dagger$.
\end{cor}

\begin{proof}
Suppose that $S\leq\lambda T$ and $(\lambda T)^\dagger=\lambda^{-1}T^\dagger\leq S^\dagger$ for some $\lambda>0$. Since $S$ and $\lambda T$ are positive, their negative parts vanish, and the additional kernel conditions of Theorem~\ref{antitonicity-igualdad de rangos} hold automatically. Therefore
$\mathcal R(S)=\mathcal R(\lambda T)=\mathcal R(T)$.
Conversely, if $\mathcal R(S)=\mathcal R(T)$, their square roots have the same range. By Douglas' theorem there exists $\lambda>0$ such that $S\leq\lambda T$. Lemma~\ref{antitonicitypositivos}, applied to $S$ and $\lambda T$, gives
$\lambda^{-1}T^\dagger\leq S^\dagger$. Finally, the last part of the assertion follows from Proposition \ref{ROL-Gre}.
\end{proof}

We finish the article by relating the antitonicity property with the Thompson component of selfadjoint operators. 
The \textit{Thompson component} of $S\in \mathcal L^s$, 
$C_S$, is the equivalence class of $S$ under the relation studied in \cite{FM-Thompson components}. If $S=S^+-S^-$ is the positive orthogonal decomposition of $S$, then the Thompson component of $S$ can be written as:
$$
C_S=\{T\in \mathcal L^s: \mathcal R((T^+)^{1/2})=\mathcal R((S^+)^{1/2}), \mathcal R((T^-)^{1/2})=\mathcal R((S^-)^{1/2})  \}.
$$
The reader is referred to \cite{FM-Thompson components} for details.
For $S,T\in\mathcal L_{cr}^s$, this description reduces to
$
T\in C_S$ if and only if  $
\mathcal R(T^+)=\mathcal R(S^+)$ and $
\mathcal R(T^-)=\mathcal R(S^-).$ Then these conditions imply  the orthogonality conditions of Corollary~\ref{antitonicityhermitianospod} and the nullspace conditions of Theorem~\ref{antitonicity-igualdad de rangos}. 

The following result can be found  in \cite[Corollary 4.3]{FM-Thompson components}.
First, recall that the polar decomposition of $T\in\mathcal{L}(\mathcal{H})$ is the factorization $T=V_T|T|$, where $V_T$ is the unique partial isometry with $\mathcal{R}(V_T)=\overline{\mathcal{R}(T)}$ which satisfies the above factorization and also verifies $\mathcal{N}(V_T)=\mathcal{N}(T)$. In particular, for $T\in \mathcal L^s$, the operator 
$U_T$ will denote the reflection defined by $U_T=V_T+P_{\mathcal N(T)}$; that is, $U_T$ satisfies $U_T=U_T^{-1}=U_T^*$.

\begin{pro}\label{Thompson component FM}
    Let $S\in\mathcal{L}^s$, then $C_S=\{T\in \mathcal L^s: \mathcal R(|T|^{1/2})=\mathcal R(|S|^{1/2}), U_T=U_S\}$.
\end{pro}

Next, we prove that  the Thompson component of a closed range operator in $\mathcal{L}^s$  can be described by selfadjoint closed range operators satisfying a weak version of the antitonicity property.  

\begin{pro}\label{Thompson component}
Let $S\in\mathcal{L}_{cr}^s$, then
\begin{equation}\label{componente autoadjunta}
C_S=\{T\in\mathcal L^s_{cr}:
|S|\leq\lambda_1T^++\lambda_2T^-,
|S|^\dagger\geq\dfrac1{\lambda_1}(T^+)^\dagger+
\dfrac1{\lambda_2}(T^-)^\dagger,
\text{for }\lambda_1,\lambda_2>0,\, U_S=U_T
\}.
\end{equation}
 In particular, if $S\in\mathcal{L}_{cr}^+$ then
 \begin{equation} \label{componente positiva}
C_S=\{T\in\mathcal{L}_{cr}^+ : \ S \leq \lambda T \ \textrm{and} \ \frac{1}{\lambda}T^\dagger \leq S^\dagger \textrm{ for some } \lambda>0\}.
\end{equation}
\end{pro}

\begin{proof}
Suppose first that $T\in C_S$. The defining equalities for the component give
$
\mathcal R\bigl((T^+)^{1/2}\bigr)
=\mathcal R\bigl((S^+)^{1/2}\bigr)=\mathcal R(S^+)
$ and $
\mathcal R\bigl((T^-)^{1/2}\bigr)
=\mathcal R\bigl((S^-)^{1/2}\bigr)=\mathcal R(S^-).
$
These ranges are closed, then $T^+$ and $T^-$ have closed ranges, and so  $T\in\mathcal L_{cr}^s$. Then it follows that
$\mathcal R(T^+)=\mathcal R(S^+)$ and $\mathcal R(T^-)=\mathcal R(S^-)$. Moreover, $U_T=U_S$ by Proposition~\ref{Thompson component FM}.
Then, by Douglas' theorem there exist positive constants $\lambda_1,\lambda_2$ such that
$S^+\leq\lambda_1T^+$ ,$S^-\leq\lambda_2T^-$. Applying Lemma~\ref{antitonicitypositivos} to each pair gives
$\frac1{\lambda_1}(T^+)^\dagger\leq(S^+)^\dagger$ and $
\frac1{\lambda_2}(T^-)^\dagger\leq(S^-)^\dagger.$
Adding these inequalities, and using the orthogonality of the positive and negative parts, proves both inequalities in \eqref{componente autoadjunta}. Conversely, suppose $T\in\mathcal L_{cr}^s$ satisfies the two inequalities in \eqref{componente autoadjunta} for some $\lambda_1,\lambda_2>0$, and $U_T=U_S$. Let $W=\lambda_1T^++\lambda_2T^-
$. Then $W$ is positive with closed range, $\mathcal R(W)=\mathcal R(T)$, and Lemma~\ref{dpo de MP} gives $
W^\dagger=\frac1{\lambda_1}(T^+)^\dagger+
\frac1{\lambda_2}(T^-)^\dagger
$. 
The assumed inequalities are $|S|\leq W$ and $W^\dagger\leq|S|^\dagger$. Applying Douglas' theorem to each inequality gives, respectively,
$
\mathcal R(S)=\mathcal R(|S|^{1/2})
\subseteq\mathcal R(W^{1/2})=\mathcal R(T),
$
and
$
\mathcal R(T)=\mathcal R\bigl((W^\dagger)^{1/2}\bigr)
\subseteq\mathcal R\bigl((|S|^\dagger)^{1/2}\bigr)=\mathcal R(S).
$
Thus $\mathcal R(S)=\mathcal R(T)$, or equivalently
$\mathcal R(|S|^{1/2})=\mathcal R(|T|^{1/2})$. Together with $U_T=U_S$, Proposition~\ref{Thompson component FM} now yields $T\in C_S$. Finally, if $S\geq0$, the defining description of $C_S$ implies that every $T\in C_S$ is positive. For positive $S,T$, their reflections both equal $I$, and $T^-=0$. Consequently \eqref{componente autoadjunta} reduces to \eqref{componente positiva}.\end{proof}

\section*{Statements and Declarations}
 The authors have no conflicts of interest.


\begin{thebibliography}{10}

\bibitem{arias2008generalized}
M.L.~Arias, G.~Corach, M.C.~Gonzalez; \emph{Generalized inverses and Douglas
  equations}. Proc. Amer. Math. Soc. 136 (9) (2008), 3177--3183.

\bibitem{AG-aditividad} M.L. Arias, G. Corach, M.C. Gonzalez; \emph{Additivity properties of operator ranges}, Linear Algebra Appl. 439 (2013),   3581--3590.


\bibitem{MR1048801}
J.K. Baksalary, K. Nordstr\"{o}m, G.P.H. Styan; \emph{L\"{o}wner-ordering antitonicity of generalized inverses of {H}ermitian matrices}. Linear Algebra  Appl.  127 (1990), 171--182. 


\bibitem{BHdSW-square integrable}
J. Behrndt, S. Hassi, H. De Snoo,  R. Wietsma; \emph{Square-integrable solutions and Weyl functions for singular canonical systems}, Math. Nachr. 284  (11--12) (2011), 1334--1384. 

\bibitem{BHdSW-limit}
J. Behrndt, S. Hassi, H. De Snoo, R. Wietsma; \emph{Limit properties of
monotone matrix functions}, Linear Algebra Appl. 436 (5) (2012), 935--953. 

 \bibitem{BHWDSAntitonicity} J. Behrndt, S. Hassi, H.S.V. de Snoo, H.L. Wietsma;\emph{Antitonicity of the inverse for selfadjoint matrices, operators, and relations}. 
 Proc. Amer. Math. Soc.  142 (8)  (2014), 2783--2796.



 \bibitem{Bouldin1973}
R. Bouldin; \emph{The pseudo-inverse of a product}.
SIAM J. Appl. Math. 24 (4) (1973), 489--495.

\bibitem{MR0203464}
R.G.~Douglas; \emph{On majorization, factorization, and range inclusion of
  operators on {H}ilbert space}. Proc. Amer. Math. Soc. 17 (1966), 413--415.

 \bibitem{FW}P.A. Fillmore,  J.P. Williams; \emph{On operator ranges.} Adv. Math. 7 (3)(1971), 254--281.
 
  \bibitem{FG-splitting}
  G. Fongi, M.C. Gonzalez; \emph{Proper splitting of Hilbert space operators}. J. Math. Anal. Appl. 545 (2025), 129093.

 \bibitem{FG-MoorePenrose y orden}
 G. Fongi, M.C. Gonzalez; \emph{
Moore-Penrose inverse and partial orders on Hilbert space operators}.
Linear Algebra Appl. 674 (2023), 1--20. 


  \bibitem{FM-Thompson components} G. Fongi, A. Maestripieri; \emph{Differential structure of the Thompson components of selfadjoint operators}. Proc. Amer. Math. Soc. 136 (2008),  613--622.
  
\bibitem{FM-desc.positivas}
G. Fongi, A. Maestripieri; \emph{Positive decompositions of selfadjoint operators}.  Integr. Equ. Oper. Theory 67 (2010), 109--121. 



\bibitem{Gre}
T.N.E. Greville; \emph{Note on the generalized inverse of a matrix product}. SIAM Rev. 8 (4) (1966), 518--521.

 \bibitem{HassiNor1993Antitonicity} S. Hassi,  K. Nordstr\"om; \emph{Antitonicity of the inverse and J-contractivity.}  Operator Theory:
Advances and Applications, 61
(1993), 149--161.

\bibitem{Izumino1982}
S. Izumino; \emph{The product of operators with closed range and an extension of the reverse order law}.
Tohoku Math. J. (2) 34 (1) (1982), 43--52.

\bibitem{Liski} E.P. Liski; \emph{On L\"owner ordering antitonicity of matrix inversion.} Acta Mathematicae Applicatae Sinica 12 (4) (1996), 435--442.

\bibitem{MA} G. Milliken, F. Akdeniz;  \emph{A theorem on the difference of the generalized inverses of two nonnegative matrices}. Commun. Statist. Theor. Meth., A6, 6 (1) (1977), 73--79. 

\bibitem{RN} F. Riesz, B.S. Nagy; \emph{Functional Analysis}, Dover Publications, Inc., New York, 1990.

\end{thebibliography}
\end{document}